\documentclass[10pt,leqno]{article}

\usepackage{url}
\usepackage[pagebackref]{hyperref}
\usepackage{amsfonts}
\usepackage{amsthm}
\usepackage{amsmath}
\usepackage{amscd}
\usepackage{amssymb}
\usepackage[shortlabels]{enumitem}
\usepackage{mathtools}
\usepackage{tikz}
\usepackage{tikz-cd}
\usetikzlibrary{matrix,arrows}

\def\CC {{\mathbb C}}     
\def\RR {{\mathbb R}}     
\def\ZZ {{\mathbb Z}}     

\def\ring#1{\ifmmode \mathaccent'027 #1\else \rm\accent'027 #1\fi}

\def\ol  {\overline}

\def \bd {\begin{diagram}}
\def \ed {\end{diagram}}
\def\be  {\begin{eqnarray}}
\def\ee  {\end{eqnarray}}
\def\ben {\begin{eqnarray*}}
\def\een {\end{eqnarray*}}

\def\bpr {\begin{proof}[Proof]}
\def\epr {\end{proof}}
\def\bsp {\begin{split}}
\def\esp {\end{split}}
\def\bcd {\begin{CD}}
\def\ecd {\end{CD}}

\mathchardef\mhyphen="2D

\newtheorem{theorem}{Theorem}[section]
\newtheorem{thmdef}{Theorem/Definition}[section]

\newtheorem{prop}[theorem]{Proposition}
\newtheorem{coro}[theorem]{Corollary}
\newtheorem{remark}[theorem]{Remark}
\theoremstyle{definition}
\newtheorem{df}[theorem]{Definition}

\theoremstyle{plain}

\newtheorem{conj}{Conjecture}[section]

\theoremstyle{plain} \newtheorem{thm}[theorem]{Theorem}
\theoremstyle{plain} \newtheorem{lem}[theorem]{Lemma}

\numberwithin{equation}{section}

\newcommand{\Addresses}{{
  \bigskip
  \footnotesize

	(F.~Haiden) \textsc{Harvard University Department of Mathematics, Science Center, One Oxford Street, Cambridge, MA 02138, USA}\par\nopagebreak
	\textit{E-mail:} \texttt{haiden@math.harvard.edu}
	\medskip
	
	(L.~Katzarkov) \textsc{Fakult\"at f\"ur Mathematik, Universit\"at Wien, 
	Oskar-Morgenstern-Platz 1, 1090 Wien, Austria, and University of Miami, USA, and HSE Moscow, Russia}\par\nopagebreak
	\textit{E-mail:} \texttt{lkatzarkov@gmail.com}
	\medskip
	
	(M.~Kontsevich) \textsc{Institut des Hautes \'Etudes Scientifiques,
	35 route de Chartres, 91440 Bures-sur-Yvette, France}\par\nopagebreak
	\textit{E-mail:} \texttt{maxim@ihes.fr}
	\medskip
	
	(P.~Pandit) \textsc{Fakult\"at f\"ur Mathematik, Universit\"at Wien, 
	Oskar-Morgenstern-Platz 1, 1090 Wien, Austria}\par\nopagebreak
	\textit{E-mail:} \texttt{pranav.pandit@univie.ac.at}
	\medskip
}}

\begin{document}

\title{Iterated logarithms and gradient flows}
\author{F.~Haiden, L.~Katzarkov, M.~Kontsevich, P.~Pandit}
\maketitle

\begin{abstract}
We consider applications of the theory of balanced weight filtrations and iterated logarithms, initiated in previous work of the authors, to PDEs. 
The main result is a complete description of the asymptotics of the Yang--Mills flow on the space of metrics on a holomorphic bundle over a Riemann surface.
A key ingredient in the argument is a monotonicity property of the flow which holds in arbitrary dimension.
The A-side analog is a modified curve shortening flow for which we provide a heuristic calculation in support of a detailed conjectural picture.

\end{abstract}

\tableofcontents

\newpage

\section{Introduction}

The theory of center manifolds provides a powerful method of analysis of local
bifurcations in infinite-dimensional systems associated with a PDE. They allow to reduce, under certain conditions, the
infinite-dimensional dynamics near a bifurcation point to a finite-dimensional dynamics, described by a system of ordinary differential equations on the center manifold, which is the submanifold of the space of solutions  along which the PDE flow has sub-exponential growth.

A novel categorical approach to the study of center manifolds was initiated by the authors in \cite{hkkp_semistability}.
This lead to the discovery of
\begin{enumerate}[1)]
\item
The presence of \textit{iterated logarithms} in the asymptotics of solutions
\item
A theory of \textit{weight filtrations} in modular lattices
\end{enumerate}
While our previous work \cite{hkkp_semistability} provided a detailed study of the minimizing flow on the space of metrics of a quiver representation, the purpose of this work is to explore further generalizations and applications to classical PDEs, as well as suggest conjectures and directions for future research.
All this is part of more general program - Categorical K\"ahler Geometry - to which we will return with more details elsewhere.



\subsubsection*{Hermitian Yang--Mills flow}

A detailed study of the Yang--Mills functional for bundles over a compact Riemann surface $X$ with K\"ahler form $\omega$ was initiated by Atiyah--Bott~\cite{atiyah_bott}.
The limiting behavior of the gradient flow of the Yang--Mills functional on the space of connections was determined by Daskalopoulos~\cite{daskalopoulos92} and R\aa de~\cite{rade92} and is given by the associated graded of the Harder--Narasimhan--Seshadri filtration.
On the other hand, since any metric on a holomorphic bundle $E$ over $X$ has an associated metric connection compatible with the complex structure, the Yang--Mills flow can be written as
\begin{equation}
h^{-1}\partial_t h=-2i(\Lambda F-\lambda)
\end{equation}
on the space of Hermitian metrics $h$ on $E$, where $F$ is the curvature of the connected associated with $h$.
In general, $h$ will grow or decay at different rates on various subbundles in a way determined by a refinement of the Harder--Narasimhan filtration provided by the theory of weight filtrations in modular lattices (see \cite{hkkp_semistability} and Section~\ref{sec_weightfilt}).
More precisely, we show the following.

\begin{thm}
\label{thm_hym_asymp}
Let $X$ be a compact Riemann surface with K\"ahler form $\omega$ and $E$ be a holomorphic bundle on $X$.
Then there exists a canonical filtration 
\begin{equation}
0=E_0\subset E_1\subset\ldots\subset E_n=E
\end{equation}
labelled by $\beta_1<\ldots<\beta_n$ with 
\begin{equation}
\beta_k\in \RR t\oplus\RR\log t\oplus\RR\log\log t\oplus\ldots
\end{equation}
such that
\begin{equation}
\left\|\log\left(h\mid_{E_k}(x)\right)\right\|=\beta_k+O(1)
\end{equation}
where we choose some reference metric on $E$ so that $h$ becomes a positive self-adjoint section of $\mathrm{End}(E)$ and the bounded term $O(1)$ is uniform in $x\in X$.
Moreover, $E_k/E_{k-1}$ is a direct sum of stable bundles of some slope $\mu_k\in\RR$ and
\begin{equation}
\beta_k=4\pi\left(\int_X\omega\right)^{-1}\left(\mu_k-\mu(E)\right)t+\ldots.
\end{equation}
\end{thm}

The main ingredients of the proof are the theory of weight filtrations introduced in our previous work~\cite{hkkp_semistability} and a monotonicity property of the HYM flow (Theorem~\ref{prop_mono_hym}).
In the case when $X$ is a compact K\"ahler manifold of higher dimension results on the limit of the Hermitian--Yang--Mills flow on the space of connections where obtained by Daskalopoulos--Wentworth~\cite{daskalopoulos_wentworth}, Jacob~\cite{jacob15}, Collins--Jacob~\cite{collins_jacob}, Sibley--Wentworth~\cite{sibley_wentworth}.
We anticipate that our result can be extended to higher dimensional $X$, but will describe the behavior of $h$ only outside a set of complex codimension two.

\subsubsection*{Modified curve shortening flow}

In Section~\ref{sec_curve_shortening} we explore the A-side analog of the HYM flow.
The example considered is a type of curve shortening flow on a punctured cylinder.
In explicit terms, the PDE is
\begin{equation}
\partial_tf=\rho(x,f)\partial_{xx}f
\end{equation}
which is a non-linear modification of the standard 1-dimensional heat equation.
The function $\rho$ is assumed to be positive except for a finite set of quadratic zeros along the $x$-axis which are the punctures of the cylinder.

We provide a heuristic argument to show that this PDE reduces to a system of ODEs in variables $y_i=|f(x_i)|/\pi$, $i\in\ZZ/n$, of the form
\begin{equation}
\frac{\dot{y}_i}{y_i}=\frac{\epsilon_{i-1}\epsilon_i}{m_i}y_{i-1}-\left(\frac{1}{m_i}+\frac{1}{m_{i+1}}\right)y_i+\frac{\epsilon_{i}\epsilon_{i+1}}{m_{i+1}}y_{i+1}
\end{equation}
The asymptotics of this system (which depend on the $m_i$ -- distances between punctures) where completely determined in~\cite{hkkp_semistability} and are related in this case to the structure of the partially wrapped Fukaya category of the punctured cylinder.

\begin{conj}
The PDE \eqref{flow_pde} has an $n$-dimensional center manifold on which the flow is approximated by the system \eqref{system_y} in the sense that error terms of solutions are bounded in coordinates $\log(y_i)$.
\end{conj}

\subsubsection*{Acknowledgements}

We would like to thank S.~Donaldson for the encouragements and constant attention to our work. We also are also very grateful to 
P.~Griffiths, A.~Petkov, T.~Pantev, K.~Fukaya  and C.~Simpson for several illuminating discussions we have had over the last months.
We are also thankful to UMiami Applied Math seminar members 
S.~Cantrell, C.~Costner and S.G.~Ruan for useful suggestions and references.  
The authors were supported   by a Simons Investigators Grant, a Simons Collaboration   research grant, NSF DMS 150908, ERC Gemis,
DMS-1265230, DMS-1201475 OISE-1242272 PASI. Simons collaborative Grant - HMS. HSE-grant,
HMS and automorphic forms. The second author is partially supported by Laboratory of Mirror
Symmetry NRU HSE, RF government grant, ag. 14.641.31.000.



\section{Weight filtrations in modular lattices}
\label{sec_weightfilt}

This section presents a brief summary of the theory of stability in modular lattices and the weight--type filtrations introduced in our previous work~\cite{hkkp_semistability} to which we refer to for more examples and details.
We begin with some basic definitions.
A \textbf{lattice} is a partially ordered set, $L$, in which any two elements $a,b\in L$ have a least upper bound $a\vee b$ and greatest lower bound $a\wedge b$. 
The following two properties of a lattice will be crucial.
\begin{itemize}
\item 
\textbf{modularity}: $(x\wedge b)\vee (a\wedge b)=((x\wedge b)\vee a)\wedge b$
for all $x,a,b\in L$

\item
\textbf{finite length}: There is an upper bound on the length, $n$, of any chain $a_0<a_1<\ldots<a_n$ of elements in $L$.
\end{itemize}
In particular, unless $L=\emptyset$, there are least and greatest elements $0$ and $1$ in any finite length lattice.
A lattice is \textbf{artinian} if it is modular and has finite length. 
A rich class of examples to keep in mind is any collection of subspaces of a given finite-dimensional vector space which is closed under sum and intersection.

There is a good theory of slope stability for artinian lattices.
First, we need to define the analog of the Grothendieck K-group.
We use the notation $[a,b]:=\{x\in L\mid a\leq x\leq b\}$ for the interval from $a$ to $b$ in a lattice.
Given an artinian lattice, $L$, we let $K(L)$ be the abelian group with generators $\ol{[a,b]}$, $a\leq b$, and relations
\begin{equation}
\ol{[a,b]}+\ol{[b,c]}=\ol{[a,c]},\qquad \ol{[a,a\vee b]}=\ol{[a\wedge b,b]}.
\end{equation}
Denote by $K^+(L)\subset K(L)$ the sub-semigroup generated by elements $\ol{[a,b]}$, $a<b$.
The Jordan--H\"older--Dedekind theorem implies that $K(L)$ (resp. $K^+(L)$) is a free abelian group (resp. semigroup).

Stability depends on a choice of \textbf{polarization} (or \textit{central charge}), which is a group homomorphisms $Z:K(L)\to\CC$ such that the image of $K^+(L)$ is contained in the right half--plane $\{\mathrm{Re}(z)>0\}$.
Then for each $a<b\in L$ we get a well-defined \textit{phase}
\begin{equation}
\phi([a,b]):=\mathrm{Arg}(Z(\ol{[a,b]}))\in (-\pi/2,\pi/2).
\end{equation}
A polarized lattice is \textbf{semistable} if $\phi([0,x])\leq\phi(L)$ for any $x\neq 0$.
Any polarized artinian lattice breaks up canonically into semistable ones.
More precisely, there is a unique chain $0=a_0<a_1<\ldots<a_n=1$, called the \textbf{Harder--Narasimhan filtration},
such that $[a_{k-1},a_k]$ is semistable for $k=1,\ldots,n$ and
$\phi([a_0,a_1])>\ldots>\phi([a_{n-1},a_n])$.
We remark that semistability imposes no restrictions on the underlying artinian lattice since we could have chosen $Z$ purely real, for instance.

\begin{remark}
While slope stability is often introduced in the context of abelian categories, it really depends only on the lattice structure. 
Furthermore, there are many natural examples of modular lattices which do not come from abelian categories.
\end{remark}

An $\RR$-filtration in $L$ is given by real numbers $\lambda_1<\ldots<\lambda_n$ and a chain $0=a_0<a_1<\ldots<a_n=1$. 
We think of $\lambda_k$ as being associated with the interval $[a_{k-1},a_k]$.
A lattice is \textbf{complemented} if any $a\in L$ has a \textbf{complement}, i.e. an element $b\in L$ with $a\wedge b=0$, $a\vee b=1$.
This corresponds to the notion of \textit{semisimplicity} in representation theory.
An $\RR$-filtration $0=a_0<a_1<\ldots<a_n=1$ labeled by $\lambda_1<\ldots<\lambda_n$ is \textbf{paracomplemented} if for any $1\leq k\leq l\leq n$ with $\lambda_l-\lambda_k<1$ the interval $[a_{k-1},a_l]$ is complemented.
Fixing such an $\RR$-filtration, there is another artinian lattice whose elements are given by choices of $b_k\in[a_{k-1},a_k]$, $k=1,\ldots,n$ such that $[b_k,b_l]$ is complemented for $1\leq k<l\leq n$ with $\lambda_k-\lambda_k\leq 1$.
Denote this lattice by $\mathcal M(a,\lambda)$.
If furthermore $L$ has an $\RR$-valued polarization $X:K_0(L)\to\RR$ then $\mathcal M(a,\lambda)$ can be given the polarization
\begin{equation}
Z([b,c]):=\sum_{k=1}^n(1+i\lambda_k)X([b_k,c_k]).
\end{equation}

The following theorem, proven in \cite{hkkp_semistability}, provides a canonical way of breaking up an artinian lattice with $\RR$-polarization into complemented ones.

\begin{thmdef}\label{balanced_chain_thm}
Let $L$ be an artinian lattice and $X:K(L)\to\RR$ an $\RR$-valued polarization.
Then there exists a unique paracomplemented $\RR$-filtration $(a,\lambda)$, called the \textbf{weight filtration}, such that $\mathcal M(a,\lambda)$ is semistable with phase $\phi=0$.
\end{thmdef}

Note that the weight filtration is trivial (i.e. $n=1$, $\lambda_1=0$) if and only if $L$ is complemented.
There is a sublattice
\begin{equation}\label{sublattice_0}
\mathcal M(a,\lambda)^0:=\{x\in \mathcal M(a,\lambda)\mid x=0\text{ or }\phi([0,x])=0\}
\end{equation}
and $Z$ restricts to an $\RR$--valued polarization on it.
If this lattice is not complemented (which corresponds to the case when $L$ is \textit{polystable}) we may apply the theorem again and get a refined filtration of $L$ indexed by $\RR^2$ with the lexicographical order.
Proceeding inductively until reaching a complemented lattice we get the \textbf{iterated weight filtration} indexed by the space of functions
\begin{equation}
\RR^\infty=\RR\log t\oplus \RR\log\log t\oplus\ldots
\end{equation}
ordered by growth as $t\to+\infty$.
Examples in \cite{hkkp_semistability} show that arbitrarily deep refinement can occur.

\section{K\"ahler type DG--algebras}
\label{sec_lozenge}

In this section we introduce an algebraic framework which unifies both
\begin{enumerate}[1)]
\item
$U(n)$ Yang--Mills bundles over Riemann surfaces and
\item
Quiver representations with harmonic metric.
\end{enumerate}
This framework will be used to construct asymptotic solutions to certain gradient flows by an iterative procedure.
It is a generalization and reinterpretation of the $*$-bimodule formalism used in \cite{hkkp_semistability}.

\subsection{Motivation}

Let $X$ be a compact Riemann surface with K\"ahler form $\omega$ and $E$ a holomorphic bundle which is a direct sum of stable bundles of various slopes. 
According to a theorem of Narasimhan--Seshadri~\cite{narasimhan_seshadri}, there exists a Hermitian metric on $E$ such that the corresponding compatible connection has constant central curvature $F$, i.e. satisfies the Yang--Mills equation $d^*F=0$.
A detailed study of the Yang--Mills functional for bundles over a Riemann surface was initiated by Atiyah--Bott~\cite{atiyah_bott}.
The deformation theory of $E$ is controlled by the DG-algebra 
\begin{equation}
A:=\mathcal A^\bullet(X,\mathrm{End}(E))
\end{equation}
of forms with values in the endomorphism bundle of $E$. 
The differential on $A$ is induced by the connection on $E$ and the product is a combination of the wedge product of forms and the composition of endomorphisms.
Moreover $A$ has a trace
\begin{equation}
\tau:A^2\to\mathbb C,\qquad \tau(\alpha)=\int_X\mathrm{tr}(\alpha),
\end{equation}
a $*$-structure $\alpha\mapsto\alpha^*$ coming from the metric on $E$, and a splitting $A^1=A^{1,0}\oplus A^{0,1}$ coming from the complex structure on $X$.
The K\"ahler form $\omega$ on $X$ gives an isomorphism $L:A^0\to A^2$, $\alpha\mapsto \omega\wedge\alpha$.
The condition on the metric on $E$ implies the relation
\begin{equation}
\Delta=2\Delta_\partial=2\Delta_{\overline{\partial}}
\end{equation}
among the Laplacians for the differentials $d,\partial,\overline{\partial}$ on $A$.
We will encounter finite--dimensional instances of algebras with these kinds of structures in our analysis of the Yang--Mills flow. 
For this reason we summarize the relevant axiomatics below.

\subsection{Lozenge algebras}

A \textbf{curved DG-algebra} over a field $\mathbf{k}$ is a $\ZZ$-graded associative algebra, $A$, with unit $1$, $\mathbf{k}$-linear derivation $d:A^\bullet\to A^{\bullet+1}$, and element $\theta\in A^2$, the \textit{curvature}, such that $d\theta=0$ and $d^2a=[\theta,a]$ for any $a\in A$.
An element $\alpha\in A^1$ gives rise to a deformation $(A,\widetilde{d},\widetilde{\theta})$ of $(A,d,\theta)$ where
\begin{equation}
\widetilde{d}a:=da+[\alpha,a], \qquad \widetilde{\theta}:=\theta+d\alpha+\alpha^2.
\end{equation}
In the following we will always have $d^2=0$ or equivalently $\theta$ is central.
In this case $A$ can be considered as an ordinary DG-algebra, but $\theta$ provides additional data.

A \textbf{Calabi--Yau structure} of dimension $n$ on a DG-algebra $A$ over a field $\mathbf{k}$ with finite dimensional total cohomology $H^\bullet(A)$ is given by a linear functional
\begin{equation}
\tau:A^n\to \mathbf{k}
\end{equation}
called the \textit{trace} so that $\tau([a,b])=0$, $\tau(da)=0$ for all $a,b\in A$, and $(a,b)\mapsto \tau(ab)$ induces perfect pairings $H^k(A)\otimes H^{n-k}(A)\to\mathbf{k}$.
Here and throughout, $[a,b]$ denotes the supercommutator which is $ab-(-1)^{|a||b|}ba$ for homogeneous elements $a,b$ of degrees $|a|,|b|$ respectively.

A \textbf{$*$--structure} on a curved DG-algebra $A$ over $\CC$ is given by a $\CC$-antilinear involution $a\mapsto a^*$ such that
\begin{equation}
(ab)^*=(-1)^{|a||b|}b^*a^*,\qquad (da)^*=da^*,\qquad 1^*=1, \qquad \theta^*=-\theta.
\end{equation}
See~\cite{deligne_morgan_susy} for the categorical justification for this convention.
If $A$ has a trace then we require $\tau(a^*)=\overline{\tau(a)}$.

\begin{df}
A \textbf{lozenge algebra} is a curved DG-algebra $A$ concentrated in degrees $0,1,2$ with $d^2=0$, trace $\tau:A^2\to\CC$ giving a Calabi--Yau structure of dimension $2$, a direct sum decomposition $A^1=A^{1,0}\oplus A^{0,1}$ as an $A^0$ bimodule such that 
\begin{equation}
(A^{1,0})^*=A^{0,1},\qquad A^{1,0}A^{1,0}=0
\end{equation}
and an element $\omega\in A^2$ with 
\begin{equation}
\omega^*=\omega,\qquad [\omega,a]=0
\end{equation}
inducing an isomorphism $L:A^0\to A^2$, $a\mapsto \omega\wedge a$ with inverse $\Lambda:A^2\to A^0$.
Furthermore, we require that the following bilinear maps
\begin{gather}
\overline{A^{1,0}}\otimes A^{1,0}\to\CC,\qquad a\otimes b\mapsto -i\tau(a^*b) \\
\overline{A^{0,1}}\otimes A^{0,1}\to\CC,\qquad a\otimes b\mapsto i\tau(a^*b) \\
\overline{A^0}\otimes A^0\to\CC,\qquad a\otimes b\mapsto \tau(\omega a^*b)
\end{gather}
are positive definite, thus providing scalar products on $A^0$ and $A^1$.
Finally, the Yang--Mills condition on the curvature
\begin{equation}\label{abstract_ym_condition}
d\Lambda\theta=0
\end{equation}
should hold.
\end{df}

Suppose $A$ is a lozenge algebra, then the differential $d:A^0\to A^1$ is a sum of $\partial: A^0\to A^{1,0}$ and $\overline{\partial}: A^0\to A^{0,1}$, and similarly $d:A^1\to A^2$ is a sum of $\partial: A^{0,1}\to A^2$ and $\overline{\partial}: A^{1,0}\to A^2$.
\begin{equation} 
\begin{tikzcd}
& A^{1,0} \arrow{dr}{\overline{\partial}} \\
A^0\arrow{ur}{\partial}\arrow{dr}{\overline{\partial}} & & A^2 \\
& A^{0,1} \arrow{ur}{\partial}
\end{tikzcd}
\end{equation}
The usual first-order K\"ahler identities can be used to define adjoints of these differentials.

\begin{lem}
Adjoints of $\partial$, $\overline{\partial}$, and $d$ are given by 
\begin{equation}\label{kaehler_ident1}
\partial^*=i[\Lambda,\overline{\partial}],\qquad \overline{\partial}^*=-i[\Lambda,\partial],\qquad d^*=i[\Lambda,\overline{\partial}-\partial]
\end{equation}
respectively.
\end{lem}

\begin{proof}
We consider $\partial:A^0\to A^{1,0}$, the other cases are entirely similar.
Let $a\in A^0$ and $b\in A^{1,0}$, then
\begin{equation}
\langle \partial a,b\rangle = -i\tau\left((\partial a)^*b\right)=-i\tau(\overline{\partial} a^*b)=i\tau(a^* \overline{\partial} b)=\tau(\omega a^*\partial^*b)=\langle a,\partial^*b\rangle
\end{equation}
where we used $\tau(\overline{\partial}(a^*b))=\tau(d(a^*b))=0$.
\end{proof}

The three Laplacians 
\begin{equation}
\Delta=dd^*+d^*d,\qquad \Delta_{\overline{\partial}}=\overline{\partial}\overline{\partial}^*+\overline{\partial}^*\overline{\partial},\qquad \Delta_{\partial}=\partial\partial^*+\partial^*\partial
\end{equation}
are related by the following lemma.

\begin{lem}
\begin{equation}\label{kaehler_ident2}
\Delta=2\Delta_{\overline{\partial}}=2\Delta_{\partial}
\end{equation}
\end{lem}

\begin{proof}
First,
\begin{equation}
\Delta=(\partial^*+\overline{\partial}^*)(\partial+\overline{\partial})+(\partial+\overline{\partial})(\partial^*+\overline{\partial}^*)=\Delta_\partial+\Delta_{\overline{\partial}}.
\end{equation}
Since $d^2=0$ by assumption we get $[\partial,\overline{\partial}]=0$ (anticommutator) and
by \eqref{kaehler_ident1}:
\begin{equation}
\Delta_\partial=[\partial^*,\partial]=i[[\Lambda,\overline{\partial}],\partial]=-i[[\Lambda,\partial],\overline{\partial}]=[\overline{\partial}^*,\overline{\partial}]=\Delta_{\overline{\partial}}
\end{equation}
\end{proof}

\begin{lem}\label{lem_lozenge_semisimple}
Suppose $A$ is a lozenge algebra with $A^0$ finite--dimensional, then $A^0$ is a $C^*$-algebra, in particular semisimple, i.e. a product of matrix algebras.
\end{lem}

\begin{proof}
By assumption, $A^0$ has an inner product $\langle a,b\rangle=\tau(\omega a^*b)$. 
The (faithful) regular representation $A^0\to\mathrm{End}(A^0)$ by left multiplication of $A^0$ on itself is a $*$-representation since
\begin{equation}
\langle l_ab,c\rangle=\tau(\omega b^*a^*c)=\langle b,l_{a^*}c\rangle
\end{equation}
i.e. $(l_a)^*=l_{a^*}$.
Thus, if $A^0$ is finite--dimensional, then $A^0$ is a $*$-subalgebra of the finite--dimensional $C^*$-algebra $\mathrm{End}(A^0)$.
Any finite--dimensional $C^*$-algebra is a product of matrix algebras.
\end{proof}

The lemma can be used to classify finite--dimensional lozenge algebras.
We have 
\begin{equation}
A^0\cong\mathrm{End}(\mathcal E_1)\oplus\ldots\oplus\mathrm{End}(\mathcal E_n)
\end{equation}
for some finite--dimensional Hermitian vector spaces $\mathcal E_i$.
Any bimodule, in particular $\mathcal A^{0,1}$, is a direct sum of simple bimodules $\mathrm{Hom}(\mathcal E_i,\mathcal E_j)$.
If we associate an arrow $i\to j$ to each such simple summand we get a quiver with vertices $\{1,\ldots,n\}$.
Furthermore, since $A^0$ is semisimple,the derivation $\overline{\partial}:A^0\to A^{0,1}$ must be inner, i.e. of the form $a\mapsto [\alpha,a]$ for some $\alpha\in A^{0,1}$.
The $\mathcal E_i$ together with $\alpha$ are exactly the data of a representation of the quiver.
This shows that in the finite--dimensional case we recover precisely the setup considered in the previous paper \cite{hkkp_semistability} but using a slightly different formalism.
The table below describes the translation between the two formalisms.

\begin{center}
\begin{tabular}{|c|c|}
\hline
Bimodule formalism & Lozenge algebra \\
\hline
\hline
$B$ & $A^0\cong A^2$ via $\omega$ \\
\hline
$\overline{M}\oplus M$ & $A^1=A^{1,0}\oplus A^{0,1}$ \\
\hline
$M\otimes\overline{M}\to B$ & $(a,b)\mapsto -i\Lambda(ab)$ \\
\hline
$\overline{M}\otimes M\to B$ & $(a,b)\mapsto i\Lambda(ab)$ \\
\hline
$\rho$ & $i\theta/\omega$ \\
\hline
$\tau$ & $b\mapsto \tau(\omega b)$ \\
\hline
$b\mapsto [\phi_0,b]$ & $\overline{\partial}$ \\
\hline
\end{tabular}
\end{center}

\subsection{The harmonic part of the algebra}

Let $A$ be a lozenge algebra. 
In particular, the cohomology $H(A)$ is assumed to be finite--dimensional.
This implies that the subspace of harmonic chains
\begin{equation}\label{harmonic_part}
\mathcal H=\mathrm{Ker}(\Delta)\subset A
\end{equation}
is finite--dimensional, and as a consequence there exists an orthogonal projection $P:A\to \mathcal H$. 
The restricted operator 
\begin{equation}
\Delta\mid_{\mathcal H^\perp}:\mathcal H^\perp\to\mathcal H^\perp
\end{equation}
is then injective, but could fail to be surjective.
We assume henceforth that it is surjective, which will be true if $A$ is either finite--dimensional or by harmonic theory in the vector bundle case.
Then there is a unique \textit{Green's operator} $G:A\to A$ with
\begin{equation}\label{greens}
PG=GP=0,\qquad P+\Delta G=P+G\Delta=\mathrm{id}_A.
\end{equation}
Since $\Delta$ commutes with $d$ and $d^*$, the same is also true for $G$.

Under the assumption of existence of a Green's operator the cohomology of $A$ is isomorphic to $\mathcal H$ as a vector space, however $\mathcal H$ is in general not a subalgebra of $A$.

\begin{lem}
Let $A$ be a lozenge algebra with Green's operator, then $\mathcal H=\mathrm{Ker}(\Delta)$ is a lozenge algebra with $d=0$, the same curvature $\theta\in A^2$, $\omega\in A^2$, restricted trace $\tau\mid_{\mathcal H}$, and product $\mathcal H^1\otimes \mathcal H^1\to \mathcal H^2$ the composition
\begin{equation}
\mathcal H^1\otimes \mathcal H^1\longrightarrow A^2 \xrightarrow{P} \mathcal H^2
\end{equation}
of the product on $A$ and the projection to the harmonic subspace.
Moreover, $\mathcal H$ is isomorphic to $H(A)$ and quasi-isomorphic to $A$ (i.e. $A$ is formal).
\end{lem}

\begin{proof}
Taking into account that $A$ is concentrated in degrees $0,1,2$ and the K\"ahler identities \eqref{kaehler_ident1}, \eqref{kaehler_ident2} we see that harmonicity can be characterized by
\begin{gather}
a\in A^0: \Delta a=0 \Leftrightarrow \partial a=0 \Leftrightarrow \overline{\partial}a=0 \\
a\in A^{1,0}: \Delta a=0 \Leftrightarrow \partial^*a=0 \Leftrightarrow \overline{\partial}a=0 \\
a\in A^{0,1}: \Delta a=0 \Leftrightarrow \partial a=0 \Leftrightarrow \overline{\partial}^*a=0 \\
a\in A^2: \Delta a=0 \Leftrightarrow \partial^*a=0 \Leftrightarrow \overline{\partial}^*a=0
\end{gather}
This shows that $\mathcal H^0\subset A^0$ is a subalgebra, $\mathcal H^1\subset A^1$ is a sub-bimodule over $\mathcal H^0$, and $\mathcal H^2\subset A^2$ is also a sub-bimodule over $\mathcal H^0$ since if $a\in\mathcal H^0$, $b\in\mathcal H^2$ then
\begin{equation}
\partial^*(ab)=-i\overline{\partial}\Lambda(ab)=-i\overline{\partial}a\Lambda(b)=-ia\overline{\partial}\Lambda(b)=a\partial^*b=0
\end{equation}
thus $ab\in\mathcal H^2$.
Furthermore $L$ and $\Lambda$ restrict to inverse isomorphisms between $\mathcal H^0$ and $\mathcal H^2$ since $d^*\omega=0$.
The YM condition \eqref{abstract_ym_condition} is precisely that $\theta\in \mathcal H^2$.

One can show directly that $\mathcal H$ is an algebra or deduce this from the isomorphism with $H(A)$.
The point is that there are quasi--isomorphisms of DG-algebras
\begin{equation}
(A,d)\hookleftarrow (\mathrm{Ker}(d^c),d) \rightarrow (H(A),0)
\end{equation}
where $d^c=i(\overline{\partial}-\partial)$.
This follows from the $dd^c$-lemma as in Deligne--Griffiths--Morgan--Sullivan~\cite{dgms}.
\end{proof}

\subsection{Gauge group action and flow}

By \textit{gauge group} we mean here the group $\mathcal G$ of invertible elements in $A^0$.
If $A^0$ is finite--dimensional, then $\mathcal G$ is isomorphic to a product of general linear groups, and if $A$ comes from a vector bundle $E$ then $\mathcal G$ is the group of automorphism of $E$ as a complex vector bundle.

If $\alpha''\in A^{0,1}$ and $g\in\mathcal G$ then define $g\cdot\alpha''$ by gauge transformations
\begin{equation}
g\cdot\alpha''=g\alpha''g^{-1}-\overline{\partial}gg^{-1}.
\end{equation}
We extend this action to $A^1$ so that if $\alpha^*=-\alpha$ then $(g\cdot\alpha)^*=-g\cdot\alpha$. 
Explicitly, we get the formula
\begin{equation}
g\cdot \alpha:=g^{*-1}\alpha'g^*+g^{*-1}\partial g^{*} + g\alpha''g^{-1}-\overline{\partial} gg^{-1}
\end{equation}
for $\alpha=\alpha'+\alpha''\in A^{1,0}\oplus A^{0,1}$ and $g\in \mathcal G$.

Given $\alpha\in A^1$ the associated curvature is
\begin{equation}
F:=\theta+d\alpha+\alpha^2
\end{equation}
and under the condition $\alpha^*=-\alpha$ we have $F^*=-F$.
Given a fixed $\alpha\in A^1$ we can consider at least formally the flow 
\begin{gather}
\dot{g}g^{-1}=-i(\Lambda F-\lambda) \\
F=\theta+d(g\cdot\alpha)+(g\cdot\alpha)^2
\end{gather}
where $\dot{g}=dg/dt$ and $\lambda$ is chosen so that $\tau(\omega(\Lambda F-\lambda))=0$, i.e. $\lambda:=\tau(\theta)/\tau(\omega)$.
If $\alpha^*=-\alpha$ then the right hand side is $*$-invariant and we replace the left hand side by the Hermitian part:
\begin{equation}\label{flow_abstract_gauge}
\frac{1}{2}\left(\dot{g}g^{-1}+(\dot{g}g^{-1})^*\right)=-i(\Lambda F-\lambda)
\end{equation}
This allows multiplying $g$ by any unitary elements.

In terms of $h=g^*g$ the curvature is
\begin{gather}
g^{-1}\left(\theta+d(g\cdot\alpha)+(g\cdot\alpha)^2\right)g=\theta +d\left(A_{\alpha,h}\right)+\left(A_{\alpha,h}\right)^2 \\
A_{\alpha,h}:=\alpha''+h^{-1}\alpha'h+h^{-1}\partial h
\end{gather}
so the flow \eqref{flow_abstract_gauge} becomes
\begin{equation}\label{flow_abstract_metric}
h^{-1}\dot{h}=-2i\left(\Lambda\left(\theta+d\left(A_{\alpha,h}\right)+\left(A_{\alpha,h}\right)^2\right)-\lambda\right).
\end{equation}
If $A$ is finite--dimensional then this is the flow considered in \cite{hkkp_semistability} up to a factor of 2.
Note that in this case there is a partial order on self-adjoint elements of $A^0$ such that positive elements are those with non-negative spectrum.
For the rest of this subsection assume that $A$ is finite--dimensional unless otherwise stated. 
The following results are established in \cite{hkkp_semistability}. 
In Section~\ref{sec_hym} we will prove versions of these results in the Hermitian vector bundle case.

\begin{prop}[Monotonicity]
\label{prop_mono_quiver}
Let $g_t,h_t$ be solutions of \eqref{flow_abstract_metric} for $t\geq 0$ with $g_0\leq h_0$, then $g_t\leq h_t$ for all $t\geq 0$.
\end{prop}

\begin{coro}[Uniqueness of asymptotics]
Let $g_t,h_t$ be solutions of \eqref{flow_abstract_metric} for $t\geq 0$. 
Then there is a constant $C\geq 1$ such that
\begin{equation}
\frac{1}{C}g_t\leq h_t\leq Cg_t
\end{equation}
for $t\geq 0$.
\end{coro}

We call $h$ an \textbf{asymptotic solution} of \eqref{flow_abstract_metric} if for some (hence any) exact solution $k$ there is a constant $C\geq 1$ such that
\begin{equation}
\frac{1}{C}k_t\leq h_t\leq Ck_t
\end{equation}
for all $t\gg 0$.

\begin{prop}[Criterion for asymptotic solutions]
\label{prop_asym_criterion}
Suppose $g$ satisfies \eqref{flow_abstract_gauge} up to an error term $s$, i.e.
\begin{equation}
\frac{1}{2}\left(\dot{g}g^{-1}+(\dot{g}g^{-1})^*\right)=-i\left(\Lambda F-\lambda\right)+s
\end{equation}
where $s$ is a smooth $t$-dependent self-adjoint element of $A^0$, and furthermore
\begin{equation}
-f\leq s\leq f
\end{equation}
for some smooth $L^1$ function $f:[0,\infty)\to[0,\infty)$.
Then $h:=g^*g$ is an asymptotic solution of \eqref{flow_abstract_metric}.
\end{prop}

King's theorem~\cite{king94} relates existence of a fixed point of the flow to slope stability.
Let $A$ be a lozenge algebra, not necessarily finite--dimensional.
By assumption, $\mathrm{Ker}(d:A^0\to A^1)$ is finite--dimensional, thus a direct sum of matrix algebras $\mathrm{End}(\mathcal E_i)$ for some finite--dimensional Hermitian vector spaces $\mathcal E_i$ (proof as for Lemma~\ref{lem_lozenge_semisimple}).
The partially ordered set of collections of subspaces $V_i\subset\mathcal E_i$ for each $i$ is a modular lattice which can be equivalently described in terms of orthogonal projectors as
\begin{equation}
\mathcal M(A):=\left\{p\in A^0\mid p^2=p, p^*=p, dp=0\right\}.
\end{equation}
The partial order by inclusion translates to $p\leq q\Leftrightarrow qp=p\Leftrightarrow pq=p$.
The central elements $\theta,\omega\in A^2$ together give a polarization 
\begin{equation}
Z([p,q]):=\tau\left((\omega-\theta)(q-p)\right)
\end{equation}
of the lattice.
Note that real part of $Z$ depends on $\omega$ and the imaginary part on $\theta$.

Given $\alpha''\in A^{0,1}$ we may also consider the sublattice
\begin{equation}
\mathcal M(A,\alpha''):=\left\{p\in\mathcal M(A,\alpha)\mid (1-p)\alpha'' p=0\right\}
\end{equation}
of projectors compatible with $\alpha''$.
In contrast to $\mathcal M(A)$, this lattice need not be complemented.

\begin{df}
A polarized lattice $(L,Z)$ is \textbf{polystable} of phase $\phi\in \RR$ if $\phi(L)=\phi$, $\phi([0,x])\leq\phi(L)$ for any $x\neq 0$, and if $\phi([0,x])=\phi(L)$ then $x$ has a complement in $L$.
Equivalently, $(L,Z)$ is semistable and $L^\phi$ is complemented in the notation of \eqref{sublattice_0}.
\end{df}

\begin{thm}[King]
\label{thm_king}
Let $A$ be a finite--dimensional lozenge algebra, $\alpha=\alpha'+\alpha''\in A^{1,0}\oplus A^{0,1}$ with $\alpha^*=-\alpha$, then there exists a $g\in\mathcal G$ such that 
\begin{equation}
\Lambda\left(\theta+d(g\cdot\alpha)+(g\cdot\alpha)^2\right)=\lambda
\end{equation}
i.e. a fixed point of the flow \eqref{flow_abstract_gauge}, if and only if $\mathcal M(A,\alpha'')$ is polystable of phase $\phi=\arctan(i\lambda)$.
\end{thm}

\subsection{Constructing asymptotic solutions}

In this subsection we construct asymptotic solutions of \eqref{flow_abstract_gauge} by relating the theory of weight filtrations in modular lattices as outlines in Section~\ref{sec_weightfilt} with the properties of lozenge algebras.

If $\mathcal M(A,\alpha'')$ is semistable of phase $0$ for some $\alpha''\in A^{0,1}$, meaning that $\tau(\theta)=0$ and $\tau(i\theta p)\leq 0$ for any $p\in \mathcal M(A,\alpha'')$, then $Z$ restricts to an $\RR$-valued polarization on the sublattice $\mathcal M(A,\alpha'')^0$ of semistables of phase $0$ (see \eqref{sublattice_0}).

\begin{df}
Let $A$ be a lozenge algebra with $\mathcal M(A,\alpha'')$ semistable.
Theorem~\ref{balanced_chain_thm} provides a weight filtration of $\mathcal M(A,\alpha'')^0$ of the form $0=p_0<p_1<\ldots<p_n=1$ labeled by real numbers $\lambda_1<\ldots<\lambda_n$.
Define projectors $p_\lambda\in\mathcal M(A)$ by
\begin{equation}
p_\lambda:=\begin{cases} p_{k}-p_{k-1} & \text{ if }\lambda=\lambda_k \\ 0 & \text{ else } \end{cases}
\end{equation}
for $\lambda\in\RR$.
The \textbf{weight grading} $r\in A^0$ is given by
\begin{equation}
r:=\sum_{\lambda\in\RR}\lambda p_\lambda
\end{equation}
which is harmonic ($dr=0$) and self-adjoint ($r^*=r$) by construction.
\end{df}

The degree $\lambda$ part of $\alpha''$ is
\begin{equation}
\alpha_\lambda'':=\sum_{\mu\in\RR}p_{\mu+\lambda}\alpha''p_\mu.
\end{equation}

\begin{lem}
$\alpha_\lambda''=0$ for $\lambda>0$, i.e. $\alpha''$ is upper--triangular with respect to the $\RR$--grading
\end{lem}

\begin{proof}
We need to check that $p_{\mu+\lambda}\alpha''p_\mu=0$ for $\lambda=0$. 
Assume $p_{\mu+\lambda}\neq 0$ and $p_\mu\neq 0$, then $p_\mu=p_k-p_{k-1}$ and $p_{\mu+\lambda}=p_m-p_{m-1}$ for some $k<m$.
We have $p_k-p_{k-1}\leq p_k$ and $p_m-p_{m-1}\leq 1-p_k$ thus
\begin{equation}
p_{\mu+\lambda}\alpha''p_\mu=(p_m-p_{m-1})(1-p_k)\alpha''p_k(p_k-p_{k-1})=0
\end{equation}
where we use $p_k\in \mathcal M(A,\alpha'')$.
\end{proof}

\begin{lem}
The polarized lattice $(\mathcal M(A,\alpha''_0),Z)$ is polystable.
\end{lem}

\begin{proof}
We need to show that $\mathcal M(A,\alpha''_0)^0$ is complemented.
By definition of the weight filtration, each interval $[p_{k-1},p_k]$ in $\mathcal M(A,\alpha'')^0$ is complemented, so it suffices to show that each $p_k$ has a complement in $\mathcal M(A,\alpha''_0)^0$.
Since $p_k\in M(A,\alpha'')$ and by definition of $\alpha''_0$ we have $(1-p_k)\alpha_0''p_k=0$ but $[p_k,\alpha_0'']=0$ so $p_k\alpha_0''(1-p_k)=0$ and $1-p_k\in \mathcal M(A,\alpha''_0)^0$ is a complement of $p_k$.
\end{proof}

Thus, if $A$ is finite-dimensional then Theorem~\ref{thm_king} tells us that after applying a gauge transformation we can assume $\alpha_0:=\alpha_0''-(\alpha_0'')^*$ is harmonic in the sense that
\begin{equation}\label{alpha0_harm}
\theta+d\alpha_0+(\alpha_0)^2=0
\end{equation}
(since we assume $\alpha=0$ we have $\lambda=0$).
This means in particular if we change the differential to $b\mapsto db+[\alpha_0,b]$ and the curvature $\theta$ to $0$ then we still have a lozenge algebra.
Thus we may assume that $\alpha_0=0$, since simultaneously twisting the differential by $\alpha_0$ and removing this term from $\alpha$ does not change the flow \eqref{flow_abstract_gauge}.

The property of the weight filtration that $[a_{k-1},a_l]$ is complemented for $\lambda_l-\lambda_k<1$, $k\leq l$ implies that after applying a gauge transformation by a harmonic invertible element in $A^0$ taking a splitting to an orthogonal one we can ensure that 
\begin{equation}\label{alpha01_vanish}
\alpha''_\lambda=0\qquad \text{ for }\lambda\in(0,1).
\end{equation}
Moreover, we may assume that 
\begin{equation}\label{alpha1_harm}
\partial\alpha_\lambda''=0 \text{ for }\lambda\leq -1
\end{equation}
i.e. we have harmonic representatives of the extension classes.

We can use $r$ and $\alpha''$ to construct a new lozenge algebra, $A_\diamond$, which
consists of harmonic chains which have certain degree with respect to $r$, thought of as an $\RR$-grading.
Precisely, set
\begin{gather}
A_\diamond^0=\{a\in A^0\mid \Delta a=0,[r,a]=0\} \\
A_\diamond^{1,0}=\{a\in A^{1,0}\mid \Delta a=0,[r,a]=-a\} \\
A_\diamond^{0,1}=\{a\in A^{0,1}\mid \Delta a=0,[r,a]=a\} \\
A_\diamond^2=\{a\in A^2\mid \Delta a=0,[r,a]=0\} \\
d_\diamond=0,\qquad \theta_\diamond=-i\omega r,\qquad \tau_\diamond=\tau,\qquad \omega_\diamond=\omega.
\end{gather}
The product $A_\diamond^{1,0}\otimes A_\diamond^{0,1}\to A_\diamond^2$ is by definition the product from $A$ composed with the projection to the harmonic part, which will then be in $A_\diamond^2$ automatically.
It follows from \eqref{alpha0_harm} that $d^2=0$ in this algebra, and $\theta_\diamond$ is central by construction.
Also, $\alpha_{-1}''\in A_\diamond^{0,1}$ by \eqref{alpha1_harm}, and $\mathcal M(A_\diamond,\alpha_{-1}'')$ is semistable of phase $0$ by definition of the weight filtration/grading.

\subsubsection*{Changing time scale}

The standing assumptions are that $A$ is a lozenge algebra and $\alpha\in A^{0,1}$ such that $\alpha_0=0$, \eqref{alpha01_vanish}, and \eqref{alpha1_harm} hold.
Suppose $x(t)\in A_\diamond^0$ solves
\begin{equation}
\frac{1}{2}\left(\dot{x}x^{-1}+\left(\dot{x}x^{-1}\right)^*\right)=-i\Lambda P\left(\left(x\cdot \alpha_{-1}\right)^2\right)-r
\end{equation}
which is just the flow \eqref{flow_abstract_gauge} in $A_\diamond$ for $\alpha_{-1}:=\alpha_{-1}''-(\alpha_{-1}'')^*$.
We want to construct a solution which solves the above equation but without the $r$ term.
Set
\begin{equation}
y(t):=(2t)^{r/2}x\left(\frac{1}{2}\log(2t)\right)
\end{equation}
then
\begin{equation}
\dot{y}=\left(r(2t)^{r/2-1}x\left(\frac{1}{2}\log(2t)\right)+(2t)^{r/2}\dot{x}\left(\frac{1}{2}\log(2t)\right)(2t)^{-1}\right)
\end{equation}
and
\begin{equation}
\dot{y}y^{-1}=(2t)^{-1}\left(r+\dot{x}\left(\frac{1}{2}\log(2t)\right)\left(x\left(\frac{1}{2}\log(2t)\right)\right)^{-1}\right)
\end{equation}
where we use $[x,r]=0$.
The left hand side is thus
\begin{equation}
\frac{1}{2}\left(\dot{y}y^{-1}+\left(\dot{y}y^{-1}\right)^*\right)=(2t)^{-1}\left(r+\frac{1}{2}\left(\dot{x}x^{-1}+\left(\dot{x}x^{-1}\right)^*\right)\right).
\end{equation}
On the other hand
\begin{align}
y\cdot\alpha_{-1} &= y^{*-1}\alpha_{-1}'y^*+y\alpha_{-1}''y^{-1} \\
&= (2t)^{-r/2}\left(x^{*-1}\alpha_{-1}'x^*\right)(2t)^{r/2}+(2t)^{r/2}\left(x\alpha_{-1}''x^{-1}\right)(2t)^{-r/2} \\
&= (2t)^{-1/2}\left(x\cdot\alpha_{-1}\right)
\end{align}
since $\alpha_{-1}'$ has $r$-degree $1$ and $\alpha_{-1}''$ has $r$-degree $-1$ and thus
\begin{equation}
\Lambda P\left(\left(y\cdot \alpha_{-1}\right)^2\right)=(2t)^{-1}\Lambda P\left(\left(x\cdot \alpha_{-1}\right)^2\right).
\end{equation}
Combining this we see that indeed
\begin{equation}
\frac{1}{2}\left(\dot{y}y^{-1}+\left(\dot{y}y^{-1}\right)^*\right)=-i\Lambda P\left(\left(y\cdot \alpha_{-1}\right)^2\right).
\end{equation}
Below we will modify $y$ so that it satisfies this equation, up to terms in $L^1$, but without the orthogonal projection $P$.

\subsubsection*{Green's function correction}

Let $x,y$ be as before and consider
\begin{gather}
w:=-i\Lambda\left(\left(y\cdot\alpha_{-1}\right)^2\right) \\
z:=y\left(1+G\left(y^{-1}wy\right)\right)
\end{gather}
where $G$ is a Green's function as in \eqref{greens}.
The factor $\left(1+G(y^{-1}wy)\right)$ has no effect on the asymptotics, but is needed to get a solution up to terms on $L^1$.
Indeed, we will show that $z$ satisfies 
\begin{equation}\label{z_eqn}
\frac{1}{2}\left(\dot{z}z^{-1}+\left(\dot{z}z^{-1}\right)^*\right)=-i\Lambda \left(d\left(z\cdot \alpha_{-1}\right)+\left(z\cdot \alpha_{-1}\right)^2\right)+\text{ terms in }L^1
\end{equation}
and is thus an asymptotic solution.

We write $O(t^\beta\mathcal L)$ for terms which are $O(t^{\beta})$ up to logarithmic corrections, e.g. $O(t^{\beta}\log t)$, $O(t^\beta\log t\log\log t)$, and so on.
As $\alpha_{-1}''$ has $r$-degree $-1$ and $w$ has $r$-degree $0$ we have
\begin{equation}
y\alpha_{-1}y^{-1}=O(t^{-1/2}\mathcal L),\qquad w=O(t^{-1}\mathcal L),\qquad G(y^{-1}wy)=O(t^{-1}\mathcal L).
\end{equation}
Consequently,
\begin{equation}
\left(1+G(y^{-1}wy)\right)^{-1}=\left(1-G(y^{-1}wy)\right)+O(t^{-2}\mathcal L)
\end{equation}
and
\begin{equation}
\dot{z}=\dot{y}\left(1+G(y^{-1}wy)\right)+yO(t^{-2}\mathcal L)
\end{equation}
where the terms in $O(t^{-2}\mathcal L)$ are of $r$-degree $0$, hence
\begin{equation}
\dot{z}z^{-1}=\dot{y}y^{-1}+O(t^{-2}\mathcal L).
\end{equation}

Next, we look at the right hand side.
We have
\begin{align}
\overline{\partial}zz^{-1} &= y\overline{\partial}G(y^{-1}wy)\left(1-G(y^{-1}wy)\right)y^{-1}+O(t^{-2}\mathcal L) \\
&=\overline{\partial}G(w)+O(t^{-2}\mathcal L) 
\end{align}
and similarly
\begin{equation}
z^{*-1}\partial z^*=\partial G(w)+O(t^{-2}\mathcal L)
\end{equation}
thus using the K\"ahler identities \eqref{kaehler_ident1} and defining property of $G$ \eqref{greens} we get
\begin{align}
-i\Lambda d\left(z^{*-1}\partial z^*-\overline{\partial}zz^{-1}\right) &=-2i\Lambda\overline{\partial}\partial G(w)+O(t^{-2}\mathcal L) \\
&=- \Delta Gw+O(t^{-2}\mathcal L) \\
&=(P-1)w+O(t^{-2}\mathcal L)
\end{align}
Recall that $\alpha_0=0$, \eqref{alpha01_vanish}, thus
\begin{equation}
\alpha''=\alpha_{-1}''+\nu''
\end{equation}
where $\nu''$ collects components of $r$-degree $<-1-\epsilon$ for some $\epsilon>0$.
Hence,
\begin{align}
z\alpha''z^{-1} &= y\left(1+G\left(y^{-1}wy\right)\right)\alpha''\left(1-G\left(y^{-1}wy\right)\right)y^{-1}+O(t^{-2}\mathcal L) \\
&=y\alpha''y^{-1}+O(t^{-3/2}\mathcal L) \\
&=y\alpha_{-1}''y^{-1}+O(t^{-(1+\varepsilon)/2}\mathcal L)
\end{align}
thus, using $\partial\alpha''=\overline{\partial}\alpha'=0$,
\begin{equation}
d\left(z\alpha''z^{-1}\right)=O(t^{-3/2}\mathcal L),\qquad d\left(z^{*-1}\alpha'z^{*}\right)=O(t^{-3/2}\mathcal L) 
\end{equation}
and
\begin{equation}
\left(z^{*-1}\alpha'z^{*}+z\alpha''z^{-1}\right)^2=\left(y^{*-1}\alpha_{-1}'y^{*}+y\alpha_{-1}''y^{-1}\right)^2+O(t^{-(1+\varepsilon)/2}\mathcal L).
\end{equation}

Combing this we get
\begin{align}
-i\Lambda \left(d(z\cdot \alpha)+(z\cdot \alpha)^2\right) &=(P-1)w-i\Lambda\left((y\cdot\alpha_{-1})^2\right)+O(t^{-(1+\varepsilon)/2}\mathcal L) \\
&= -i\Lambda P\left((y\cdot\alpha_{-1})^2\right)+O(t^{-(1+\varepsilon)/2}\mathcal L)
\end{align}
which shows \eqref{z_eqn}.

\subsubsection*{Iterative procedure}

Suppose $A$ is a finite--dimensional lozenge algebra and $\alpha''\in A^{0,1}$ such that $\mathcal M(A,\alpha'')$ is semistable of phase $0$.
Let
\begin{equation}
0=p_0<p_1<\ldots<p_n=1
\end{equation}
be the iterated weight filtration of $\mathcal M(A,\alpha'')$ labeled by elements $\beta_1<\ldots<\beta_n$ in $\RR\log t\oplus\RR\log\log t\oplus\ldots$.
The procedure described in this subsection gives a finite sequence of lozenge algebras $A,A_\diamond,\ldots$ and gauge transformations $g(t)\in A^0$ for $t\gg 0$ which give a solution of the flow \eqref{flow_abstract_gauge} up to terms in $L^1$, thus an asymptotic solution by Proposition~\ref{prop_asym_criterion}. 
In this process some fixed gauge transformation was applied to $\alpha''$ to ensure harmonicity properties.
Moreover, 
\begin{equation}
\log g(t)=\frac{1}{2}\sum_{k=1}^n\beta_k(p_k-p_{k-1})+O(1)
\end{equation}
i.e. $\log g(t)$ is, up to bounded terms, diagonal with entries which are linear combinations of iterated logarithms.

\section{Flow on Hermitian vector bundles}
\label{sec_hym}

In this section we explore the consequences of our theory in the setting of holomorphic bundles over a compact K\"ahler manifold $X$.
While initially $X$ can be of arbitrary dimension, we later restrict to the case of Riemann surfaces.

\subsection{Slope-semistable coherent sheaves}

We let $X$ be a compact K\"ahler manifold and $\omega$ its positive $(1,1)$-form.
The \textit{degree} of a coherent sheaf $E$ on $X$ is by definition
\begin{equation}
\deg(E):=\int_Xc_1(E)\wedge \omega^{n-1}
\end{equation}
and depends evidently only on the K\"ahler class $[\omega]\in H^2(X;\RR)$.
Consider
\begin{equation}
Z(E):=\mathrm{rk}(E)+\deg(E)i\in\CC
\end{equation}
then $Z(E)=0$ if the support of $E$ has codimension at least $2$.
As in \cite{Meinhardt2014}, let $\mathrm{Coh}^2(X)\subset \mathrm{Coh}(X)$ denote the full subcategory of sheaves whose support has codimension $\geq 2$, and \begin{equation}
\mathrm{Coh}_{(1)}(X)=\mathrm{Coh}(X)/\mathrm{Coh}^2(X)
\end{equation}
the quotient abelian category (localization).
The category $\mathrm{Coh}_{(1)}(X)$ behaves much like the $\mathrm{Coh}(Y)$ for a curve $Y$ in that is has cohomological dimension $\leq 1$. 
We refer to Meinhardt--Partsch~\cite{Meinhardt2014} for details.

Fix $\phi\in (-\frac{\pi}{2},\frac{\pi}{2}]$ and let $\mathrm{Coh}_{(1)}^\phi(X)\subset \mathrm{Coh}_{(1)}(X)$ be the full subcategory of semistable objects, $E$ of phase $\mathrm{Arg}(Z(E))=\phi$.
This is an abelian artinian category, hence Theorem~\ref{balanced_chain_thm} provides any $E\in\mathrm{Coh}_{(1)}^\phi(X)$ with a canonical $\RR^\infty$-filtration with polystable subquotients.
Furthermore, for any $E\in\mathrm{Coh}_{(1)}(X)$ we get a canonical refinement of its Harder--Narasimhan filtration.

\begin{remark}
If $\dim_{\CC}X>1$ then slope stability is degenerate in the sense that it does not see anything in codimension at least $2$.
Instead, suppose we have a stability condition in the sense of Bridgeland~\cite{bridgeland07} on $D^b(X)$, then for each phase $\phi\in\RR$ there is an artinian abelian category of semistable objects of phase $\phi$. 
Thus we get canonical refinements of the Harder--Narasimhan filtrations for this stability condition without needing to localize by a subcategory.
It remains an open problem however to show existence of stability conditions on $D^b(X)$ for general $X$, and even more significantly to prove an analogue of the DUY theorem (Theorem~\ref{duy_thm}) for Bridgeland stability conditions.
\end{remark}

\subsubsection*{Hermitian--Einstein metrics}

Textbook references for this material are Kobayashi~\cite{kobayashi87} and Siu~\cite{siu_book}.
Suppose $E$ is a holomorphic vector bundle over $X$.
For any choice of Hermitian metric, $h$, on $E$ there exists a unique unitary connection, $D$, whose $(0,1)$-component is the natural operator $\overline{\partial}_E$ determined by the holomorphic structure on $E$. 
With respect to a local holomorphic frame of $E$ we have
\begin{equation}
D=d+h^{-1}\partial h.
\end{equation}
The curvature of $D$ is given by
\begin{equation}
F=\overline{\partial}(h^{-1}\partial h)=h^{-1}(\overline{\partial}\partial h-\overline{\partial}hh^{-1}\partial h).
\end{equation}
Let $\Lambda$ be the operator on forms which is adjoint to $L(\alpha)=\omega\wedge\alpha$.
The metric $h$ is \textit{Hermitian--Einstein} if 
\begin{equation}
\Lambda F=\lambda
\end{equation}
for some constant $\lambda$.
Since $\mathrm{tr}\left(\frac{i}{2\pi}F\right)$ represents the first Chern class we necessarily have
\begin{equation}\label{top_lambda}
\lambda=-\frac{2n\pi\deg(E)}{\mathrm{rk}(E)\int_X\omega^n}i
\end{equation}
where $n=\dim_{\CC}X$.

The following fundamental theorem is due to Narasimhan--Seshadri for $X$ a Riemann surface \cite{narasimhan_seshadri}, Donaldson for $X$ an algebraic surface \cite{donaldson85} and Uhlenbeck--Yau for general K\"ahler $X$ \cite{uhlenbeck_yau}.

\begin{thm}\label{duy_thm}
A holomorphic vector bundle admits a Hermitian--Einstein metric if and only if it is slope-polystable.
\end{thm}

Donaldson's approach to proving existence of Hermitian--Einstein metrics is to start with an arbitrary metric and follow the nonlinear heat-type flow
\begin{equation}\label{donaldson_flow}
h^{-1}\partial_th=-2i(\Lambda F-\lambda)
\end{equation}
whose solution is unique for given initial condition and exists for all positive time.

\subsection{Monotonicity and related properties}

The flow \eqref{donaldson_flow} is homogeneous in the sense that if $h$ is a solution for $t\geq 0$ and $f:[0,+\infty)\to\RR$ a smooth function, then $g:=e^{f(t)}h$ gives the same compatible connections and thus satisfies
\begin{equation}
g^{-1}\partial_tg=-2i(\Lambda F-\lambda)+\frac{df}{dt}.
\end{equation}
The Hermitian metrics on a bundle $E$ are partially ordered by the pointwise comparison
\begin{equation}
g\leq h :\Leftrightarrow g(v,v)\leq h(v,v)\qquad\text{for all }v\in TM.
\end{equation}
The flow preserves this partial order, as in the case of quiver representations.

\begin{prop}
\label{prop_mono_hym}
Let $g_t,h_t$ be solutions of \eqref{donaldson_flow} for $t\geq 0$ with $g_0\leq h_0$, then $g_t\leq h_t$ for all $t\geq 0$.
\end{prop}

\begin{proof}
It suffices to show that for arbitrary $\varepsilon>0$ we have $k_t:=e^{-\varepsilon t}g_t\leq h_t$ for $t\in[0,\infty)$.
Furthermore, since the set of $t\in[0,\infty)$ where $k_t\leq h_t$ is closed and contains $0$, the claim follows from existence a $\delta>0$ such that $k_t\leq h_t$ for $t\in[0,\delta)$.

Let $S(TX)$ denote the unit sphere bundle inside the tangent bundle $TX$ with respect to the metric $h_0$.
We claim that each $v\in S(TX)$ has a neighborhood $U\subset S(TX)$ and $\delta>0$ such that $k_t(w,w)\leq h_t(w,w)$ for $w\in U$ and $t\in[0,\delta)$.
If $k_0(v,v)<h_0(v,v)$ this is clear, so suppose $k_0(v,v)=h_0(v,v)=1$.
Since $g,h$ are solutions of the flow we have
\begin{gather}
k^{-1}\partial_tk=-2i\left(\Lambda k^{-1}(\overline{\partial}\partial k-\overline{\partial}kk^{-1}\partial k)-\lambda\right)-\varepsilon \\
h^{-1}\partial_th=-2i\left(\Lambda h^{-1}(\overline{\partial}\partial h-\overline{\partial}hh^{-1}\partial h)-\lambda\right). 
\end{gather}
For suitable choice of local holomorphic frame of $E$ we may assume that $\partial h_0=0$ at the basepoint, $x\in X$, of $v$ (see \cite{kobayashi87}, Proposition 1.4.20).
Since $2i\Lambda\overline{\partial}\partial=\Delta$ is the geometer's Laplacian on functions we get
\begin{align}\label{mono_main_calc}
\left.\frac{d}{dt}\right\vert_{t=0}&\left(h_t(v,v)-k_t(v,v)\right)= \\
&=-\Delta\left(h_0(v,v)-k_0(v,v)\right)-2i\left(\Lambda(\overline{\partial}k_0k_0^{-1}\partial k_0)\right)(v,v)+\varepsilon \geq \varepsilon \nonumber
\end{align}
where we have used that $-i\Lambda\overline{\varphi}\wedge\varphi\geq 0$ for a $(1,0)$-form $\varphi$.
This shows existence of the desired neighborhood $U$ and $\delta>0$.
The proposition then follows from compactness of $S(TX)$.
\end{proof}

\begin{coro}
Let $g_t,h_t$ be solutions of \eqref{donaldson_flow} for $t\geq 0$. 
Then there is a constant $C\geq 1$ such that
\begin{equation}
\frac{1}{C}g_t\leq h_t\leq Cg_t
\end{equation}
for $t\geq 0$.
\end{coro}

\begin{proof}
Since $X$ is compact we can find $C\geq 1$ such that $C^{-1}g_0\leq h_0\leq Cg_0$. Apply Proposition~\ref{prop_mono_hym}.
\end{proof}

We call $h$ an \textbf{asymptotic solution} of \eqref{donaldson_flow} if for some (hence any) exact solution $k$ there is a constant $C\geq 1$ such that
\begin{equation}
\frac{1}{C}k_t\leq h_t\leq Ck_t
\end{equation}
for all $t\gg 0$.
An easily verifiable sufficient criterion for recognizing asymptotic solutions will be given below.

One may rewrite the flow as taking place in the gauge group instead of the space as metrics.
In this case there is a fixed reference metric $H$ on $E$ so that if $g$ is a smooth section $GL(E)$ then $h(v,w):=H(gv,gw)$ defines another metric on $E$, and conversely $g$ is determined by $h$ up to unitary transformations.
In terms of $g$ instead of $h$, the equation~\ref{donaldson_flow} becomes
\begin{equation}\label{gauge_flow}
\frac{1}{2}\left(\partial_tgg^{-1}+\left(\partial_tgg^{-1}\right)^*\right)=-i\left(\Lambda F-\lambda\right)
\end{equation}
where $F$ is the curvature of the metric connection
\begin{equation}
g\circ\overline{\partial}_E\circ g^{-1}+g^{*-1}\circ\partial_E\circ g^{*}.
\end{equation}

\begin{prop}\label{prop_asym_criterion_hym}
Suppose $g$ satisfies \eqref{gauge_flow} up to an error term $s$, i.e.
\begin{equation}
\frac{1}{2}\left(\partial_tgg^{-1}+\left(\partial_tgg^{-1}\right)^*\right)=-i\left(\Lambda F-\lambda\right)+s
\end{equation}
where $s$ is a smooth $t$-dependent self-adjoint section of $\mathrm{End}(E)$, and furthermore
\begin{equation}
-f\leq s\leq f
\end{equation}
for some smooth $L^1$ function $f:[0,\infty)\to[0,\infty)$.
Then the corresponding $t$-dependent metric, $h:=H(g\_,g\_)$, is an asymptotic solution of \eqref{donaldson_flow}.
\end{prop}

\begin{proof}
Let $k$ be an exact solution of \eqref{donaldson_flow}, then 
\begin{equation}
k^{\pm}_t:=\exp\left(\pm\int_0^tf\right)k_t
\end{equation}
satisfy \eqref{donaldson_flow} up to an error term $\pm f$, respectively.
Furthermore, if $C:=\exp\left(\int_0^\infty f\right)$ then $k^+\leq Ck$ and $C^{-1}k\leq k^-$.
Fix the initial condition for $k$ by $k_0=h_0=H(g_0\_,g_0\_)$.
We will show that
\begin{equation}
k^-_t\leq h_t\leq k^+_t
\end{equation}
for $t\geq 0$, which implies the proposition.

The proof is a slight modification of the one for monotonicity, Proposition~\ref{prop_mono_hym}.
The family of metrics $h$ satisfies \eqref{donaldson_flow} up to an error term $g^{-1}sg$.
In the main calculation \eqref{mono_main_calc} a new term 
\begin{equation}
f(0)k_0^+(v,v)-H(g_0v,s(0)g_0v)=H(g_0v,(f-s)(0)g_0v)\geq 0
\end{equation}
appears. Since it is non-negative, the rest of the argument still works.
\end{proof}

\subsection{Asymptotic solutions}

We restrict to the case $\dim_{\CC}(X)=1$.
Let $E$ be a holomorphic bundle over $X$ with Harder--Narasimhan filtration
\begin{equation}
0=E_0\subset E_1\subset \ldots\subset E_m=E
\end{equation}
which means that each $S_k:=E_k/E_{k-1}$ is semistable and the slopes $\mu(S_k)=\deg(S_k)/\mathrm{rk}(S_k)$ satisfy
\begin{equation}
\mu(S_1)>\mu(S_2)>\ldots>\mu(S_m).
\end{equation}

\subsubsection*{Case without refinement}

To illustrate the strategy in a simple special case where the theory of Section~\ref{sec_lozenge} is not needed, let us assume first that each $S_k$ is in fact polystable, i.e. a direct sum of stable bundles.
Then by the Narasimhan--Seshadri theorem ($n=1$ case of Theorem~\ref{duy_thm}) the associated graded bundle
\begin{equation}
S:=\bigoplus_{k=1}^m S_k
\end{equation}
admits a Hermitian metric so that the associate connection $D$ has constant central curvature $F_S$, or equivalently is a critical point of the Yang--Mills functional $\|F\|^2$.
We assume such a metric is chosen on $S$.
The holomorphic bundle $E$ is isomorphic to $S$ with modified complex structure $\overline{\partial}_S+\alpha$ for some $(0,1)$-form $\alpha$ with values in $\mathrm{End}(S)$ which is strictly block upper--triangular with respect to the direct sum decomposition $S:=\bigoplus S_k$.

An asymptotic solution of \ref{gauge_flow} is simply
\begin{equation}\label{HN_only_solution}
g_t=e^{C(\mu-\mu_S)t}
\end{equation}
where $C:=2\pi/\int_X\omega$, $\mu:=\mu(E)$, and $\mu_S$ acts by multiplication by $\mu_k$ on $S_k$.
We need to check that this satisfies \ref{gauge_flow} up to terms in $L^1$.

An element $g$ of the gauge group (smooth automorphism of $S$) acts on $\alpha$ and $\alpha^*$ by
\begin{equation}\label{gauge_action}
g\cdot\alpha:=g\alpha g^{-1}-\overline{\partial}_Sgg^{-1},\qquad g\cdot\alpha^*:=g^{*-1}\alpha^*g^*+g^{*-1}\partial_Sg^{*}
\end{equation}
and thus the connection $\overline{\partial}_S+g\cdot\alpha+\partial_S-g\cdot\alpha^*$ has curvature
\begin{equation}
F=F_S+\overline{\partial}_S(g\cdot\alpha)-\partial_S(g\cdot\alpha^*)-g\cdot\alpha\wedge g\cdot\alpha^*-g\cdot\alpha^*\wedge g\cdot\alpha.
\end{equation}
If $g$ as in \ref{HN_only_solution} then $\partial_S g=\overline{\partial}_S g=0$ and we claim that $g\alpha g^{-1}$ is exponentially decaying. 
Indeed, the $(j,k)$ block of $g\alpha g^{-1}$ is
\begin{equation}
e^{C(\mu-\mu_j)t}\alpha_{jk}e^{-C(\mu-\mu_k)t}=e^{C(\mu_k-\mu_j)t}\alpha_{jk}
\end{equation}
but $\alpha_{jk}=0$ for $j\geq k$ and $\mu_k-\mu_j<0$ for $j<k$ by assumption.
Thus
\begin{equation}
F=F_S+(\textit{exponentially decaying terms})
\end{equation}
and by assumption on the metric on $S$ we have $-i\Lambda F_S=-C\mu_S$.
Since $\partial_tgg^{-1}=C(\mu-\mu_S)$ and $i\lambda=C\mu$ (see \eqref{top_lambda}) we find that $g$ solves \eqref{gauge_flow} up to exponentially decaying terms.
By Proposition~\ref{prop_asym_criterion_hym} we see that $g$ is an asymptotic solution.
This proves Theorem~\ref{thm_hym_asymp} in this special case.

\subsubsection*{Case with refinement}

We start by labeling of the Harder--Narasimhan filtration by exponential growth rate
\begin{equation}
\beta_k:=2\pi\left(\int_X\omega\right)^{-1}(\mu(E)-\mu_k)
\end{equation}
instead of slope $\mu_k=\mu(E_k/E_{k-1})$.
For each semistable component $S_k=E_k/E_{k-1}$ there is an artinian lattice of subbundles of $S_k$ of the same slope $\mu_k$ which has an $\RR$-valued polarization given by rank.
This is up to a constant factor (which does not affect the weight filtration) the same as the polarization coming from the restriction of $Z(E)=\mathrm{rk}(E)+\deg(E)i$, since slope is fixed.
Thus we get a canonical iterated weight filtration 
\begin{equation}
0=S_{k,0}\subset S_{k,1}\subset\ldots\subset S_{k,r_k}=S_k
\end{equation}
of $S_k$ indexed by elements $\beta_{k,1}<\ldots<\beta_{k,r_k}$ in $\RR\log t\oplus\RR\log\log t\oplus\ldots$ and whose subquotients are direct sums of stable bundles of slope $\mu_k$.
These filtrations of $S_1,\ldots,S_m$ together provide a refinement of the Harder--Narasimhan filtration indexed by elements 
\begin{equation}
\beta_1t+\beta_{1,1}<\beta_1t+\beta_{1,2}<\ldots<\beta_1t+\beta_{1,r_1}<\beta_2t+\beta_{2,1}<\ldots<\beta_mt+\beta_{m,r_m}
\end{equation}
in $\RR t\oplus\RR\log t\oplus\RR\log\log t\oplus\ldots$.

For each semistable bundle, $S_k$, there is an associated graded holomorphic bundle, $T_k$, of the iterated weight filtration which is a direct sum of stable bundles of the same slope and thus admits a metric so that the associated connection has constant central curvature
\begin{equation}
F=-\frac{2\pi i\omega}{\int_X\omega}.
\end{equation}
Fixing such a metric, the algebra of endomorphism valued forms
\begin{equation}
\mathcal A^{\bullet}(X,\mathrm{End}(T_k))
\end{equation}
is a lozenge algebra with curvature $\theta=F$.
The bundles $S_k$ and $T_k$ have the same underlying complex vector bundle, but the holomorphic structure differ by
\begin{equation}
\alpha_k'':=\overline{\partial}_{T_k}-\overline{\partial}_{S_k}\in\mathcal A^{0,1}(X,\mathrm{End}(T_k))
\end{equation}
and $\alpha_k''$ is strictly upper--triangular with respect to the grading on $T_k$ coming from the iterated weight filtration on $S_k$.

Let $g_k(t)\in \mathcal A^0(X,GL(T_k))$ be a solution of the flow
\begin{equation}
\frac{1}{2}\left(\dot{g}_kg_k^{-1}+\left(\dot{g}_kg_k^{-1}\right)^*\right)=-i\Lambda\left(d\left(g_k\cdot\alpha_k\right)+\left(g_k\cdot\alpha_k\right)^2\right)
\end{equation}
where $\alpha_k=\alpha_k''-(\alpha_k'')^*$. 
(Note that $\Lambda\theta=\lambda$ cancel.)
The procedure described in Section~\ref{sec_lozenge} gives an asymptotic solution which shows that for the trajectory of metrics $h_k:=(g_k)^*g_k$ one has asymptotic growth
\begin{equation}
\left\|\log\left(h_k(t)\mid S_{k,j}\right)\right\|=\beta_{k,j}+O(1).
\end{equation}

Consider $T:=\bigoplus_k T_k$ which is the associated graded of the total filtration on $E$. 
In particular, $T$ and $E$ have the same underlying complex vector bundle and the holomorphic structures differ by
\begin{equation}
\alpha'':=\overline{\partial}_E-\overline{\partial}_T\in\mathcal A^{0,1}(X,\mathrm{End}(T))
\end{equation}
which is strictly upper triangular with respect to the grading on $T$ coming from the total filtration on $E$.
Moreover, the $T_k$ diagonal block of $\alpha$ is just $\alpha_k$.
As before one shows that
\begin{equation}
g(t):=\mathrm{diag}\left(e^{\beta_1t}g_1(t),\ldots,e^{\beta_mt}g_m(t)\right)\qquad \in\mathcal A^0(X,GL(T))
\end{equation}
solves
\begin{equation}
\frac{1}{2}\left(\dot{g}g^{-1}+\left(\dot{g}g^{-1}\right)^*\right)=-i\left(\Lambda\left(\theta+d\left(g\cdot\alpha\right)+\left(g\cdot\alpha\right)^2\right)-\lambda\right)
\end{equation}
where $\alpha=\alpha''-(\alpha'')^*$, up to terms in $L^1$.
If we set $h:=g^*g$ then
\begin{equation}
\left\|\log\left(h(t)\mid S_{k,j}\right)\right\|=2\beta_kt+\beta_{k,j}+O(1).
\end{equation}
where we consider $S_{k,j}$ also as holomorphic subbundle of $E$ by taking the preimage under $E_k\to S_k$. 
This completes the proof of Theorem~\ref{thm_hym_asymp}.

\section{Modified curve shortening flow on a cylinder}
\label{sec_curve_shortening}

On a Riemann surface with non-vanishing holomorphic 1-form $\Omega$ and symplectic form $\omega$ there is a natural curve shortening flow
\begin{equation}\label{arg_flow}
\dot{c}=-d\mathrm{Arg}\left(\Omega|_c\right)\llcorner\omega^{-1}
\end{equation}
which decreases length measured with respect to $|\Omega|^2$. 
By definition, the flow deforms the curve by (local) Hamiltonian isotopies, provided we restrict to those curves which admit global choice of $\mathrm{Arg}\left(\Omega|_c\right)$.

Motivated by the theory of stability in partially wrapped Fukaya categories of Riemann surfaces~\cite{hkk} we consider the following special case.
The flat surface is the cylinder
\begin{equation}
X=\CC/L\ZZ, \qquad \Omega=dz
\end{equation}
with circumference $L>0$.
The cylinder is punctured at points $x_0<x_1<\ldots<x_{n-1}\in [0,L)$.
The symplectic form is
\begin{equation}
\omega=\frac{dx\wedge dy}{\rho}
\end{equation}
where $\rho$ should be chosen so that near $x_i$ it is of the form
\begin{equation}
\rho(x,y)=(x-x_i)^2+y^2
\end{equation}
for simplicity, and $\rho$ vanishes at no other points.
This means that symplectically there is a half-infinite cylinder around each puncture $x_i$.

Assume the curve is a graph
\begin{equation}
x\mapsto (x,f(x)), \qquad f(x_i)\neq 0
\end{equation}
where we can consider $f$ as $L$-periodic function.
Also assume that not all $f(x_i)$ have the same sign, i.e. the curve passes both over and under some punctures (see Figure~\ref{plane_curve}).
Approximating $\mathrm{Arg}$ by the slope, the flow \eqref{arg_flow} becomes
\begin{equation}\label{flow_pde}
\partial_tf=\rho(x,f)\partial_{xx}f
\end{equation}
which is the PDE we want to study.

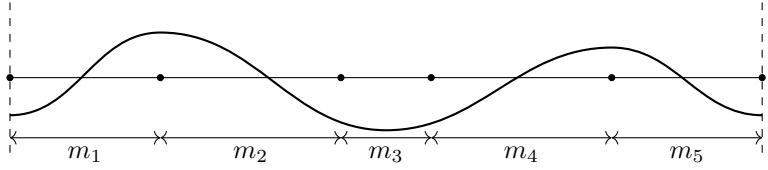
\begin{figure}
\center
\begin{tikzpicture}[scale=2]
\draw (0,0) to (5,0);
\draw[dashed] (0,-.5) to (0,.5);
\draw[dashed] (5,-.5) to (5,.5);
\foreach \x in {0, 1, 2.2, 2.8, 4,5}
 \draw[fill] (\x,0) circle [radius=0.02];
\draw [thick] (0,-.25) to [out=0,in=180] (1,.3) to [out=0,in=180] (2.5,-.35) to [out=0,in=180] (4,.2) to [out=0,in=180] (5,-.25);
\path[<->] (0,-.4) edge node[below]{$m_1$} (1,-.4);
\path[<->] (1,-.4) edge node[below]{$m_2$} (2.2,-.4);
\path[<->] (2.2,-.4) edge node[below]{$m_3$} (2.8,-.4);
\path[<->] (2.8,-.4) edge node[below]{$m_4$} (4,-.4);
\path[<->] (4,-.4) edge node[below]{$m_5$} (5,-.4);
\end{tikzpicture}
\caption{Curve on punctured flat cylinder.}
\label{plane_curve}
\end{figure}

We will give a heuristic argument to show that asymptotically, on a center manifold, this PDE reduces to a system of ODEs in variables $y_i=|f(x_i)|/\pi$, $i\in\ZZ/n$, of the form
\begin{equation}\label{system_y}
\frac{\dot{y}_i}{y_i}=\frac{\epsilon_{i-1}\epsilon_i}{m_i}y_{i-1}-\left(\frac{1}{m_i}+\frac{1}{m_{i+1}}\right)y_i+\frac{\epsilon_{i}\epsilon_{i+1}}{m_{i+1}}y_{i+1}
\end{equation}
where $\epsilon_i=\mathrm{sign}(f(x_i))$ and $m_i>0$ is the length of the segment from $x_{i-1}$ to $x_i$.
This is a special case of the system constructed from a directed acyclic graph in~\cite{hkkp_semistability}.
The calculations below provide good evidence for the following conjecture.
\begin{conj}
The PDE \eqref{flow_pde} has an $n$-dimensional center manifold on which the flow is approximated by the system \eqref{system_y} in the sense that error terms of solutions are bounded in coordinates $\log(y_i)$.
\end{conj}

The starting point is the following ansatz for $f$.
Assume that near $x_i$, $f$ is approximately of the form
\begin{equation}\label{ansatz_pt}
f(x)\sim a_{i,0}+a_{i,1}(x-x_i)+a_{i,2}\varphi\left(\frac{x-x_i}{a_{i,0}}\right)
\end{equation}
and on the segment between $x_i$ and $x_{i+1}$ we have
\begin{equation}\label{ansatz_seg}
f(x)\sim b_{i,0}(x-x_i)+b_{i,1}(x_{i+1}-x)+b_{i,2}\chi(x)+b_{i,3}\psi(x).
\end{equation}
The function $\varphi$ is the unique solution to
\begin{equation}
(x^2+1)\varphi''=1,\qquad\varphi(0)=\varphi'(0)=0 
\end{equation}
which is
\begin{equation}
\varphi(x)=x\arctan x-\frac{1}{2}\log(1+x^2)\sim \frac{\pi}{2}|x|-\log|x|
\end{equation}
where the approximation is good for $|x|\gg 0$.
The functions $\chi,\psi$ are chosen so that
\begin{equation}
\rho(x,0)\chi''=x-x_i,\qquad \rho(x,0)\psi''=x_{i+1}-x.
\end{equation}

Compatibility of \eqref{ansatz_pt} and \eqref{ansatz_seg} gives the following set of equations.
\begin{center}
\begin{tabular}{c|c|c}
term         & left of $x_i$        & right of $x_i$ \\ \hline  
$1$        & $m_ib_{i-1,0}=a_{i,0}$ & $m_{i+1}b_{i,1}=a_{i,0}$ \\
$x-x_i$    & $b_{i-1,0}-b_{i-1,1}=a_{i,1}-\frac{\pi a_{i,2}}{2|a_i,0|}$ & $b_{i,0}-b_{i,1}=a_{i,1}+\frac{\pi a_{i,2}}{2|a_i,0|}$ \\
$\log|x-x_i|$  & $m_ib_{i-1,2}=a_{i,2}$ & $m_{i+1}b_{i,3}=a_{i,2} $
\end{tabular}
\end{center}
Furthermore, plugging the ansatz into the equation \eqref{flow_pde} implies
\begin{equation}
\dot{a}_{i,0}=a_{i,2},\qquad \dot{b}_{i,0}=b_{i,2},\qquad \dot{b}_{i,1}=b_{i,3}.
\end{equation}
Combining the above gives
\begin{align}
\pi\frac{\dot{a}_{i,0}}{|a_{i,0}|}&=b_{i-1,1}-b_{i-1,0}-b_{i,1}+b_{i,0} \\
&=\frac{a_{i-1,0}}{m_i}-\frac{a_{i,0}}{m_i}-\frac{a_{i,0}}{m_{i+1}}+\frac{a_{i+1,0}}{m_{i+1}}
\end{align}
which is \eqref{system_y} with $y_i=\epsilon_ia_{i,0}/\pi$.

The signs $\epsilon_i$ determine an orientation of the graph which is a single cycle with $n$ edges.
Namely, vertices correspond to segments $[x_{i-1},x_i]$ and arrows correspond to punctures $x_i$ with orientation given by $\epsilon_i$, i.e. depending on whether the curve passes over or under the puncture.
By assumption not all arrows are oriented the same way.
More conceptually, the abelian category of representations of this directed graph appears as a full subcategory containing the object corresponding to the curve $(x,f(x))$ of the partially wrapped Fukaya category of the punctured cylinder.

The directed graph gives rise to a modular (in fact distributive) lattice $\mathcal M$ whose elements are subsets $S$ of the set of vertices so that no arrows lead out of $S$.
According to the theory developed in the previous paper~\cite{hkkp_semistability}, the asymptotics of solutions to \eqref{system_y} are determined by the iterated weight filtration on the lattice $\mathcal M$. 
Since the lattice comes from a directed graph, the simpler definition of a \textit{weight grading} may also be used.

The first case with wall--crossing is when $n=5$ with three arrows pointing one way and two the other, which corresponds to Figure~\ref{plane_curve}.
\begin{center}
\begin{tikzcd}
\underset{m_1}{\bullet} \arrow{r}\arrow[bend left]{rrrr} & \underset{m_2}{\bullet} & \underset{m_3}{\bullet} \arrow{l} & \underset{m_4}{\bullet} \arrow{l} \arrow{r} & \underset{m_5}{\bullet}
\end{tikzcd}
\end{center}
There are three chambers in the space $\RR_{>0}^5$ of parameters $m_i$.
The two disjoint walls (where $\log\log t$ appears in the asymptotics) are given by equations $D_1=0$, $D_2=0$ where
\begin{align}
D_1&=m_1m_4+m_3m_5+2m_4m_5-m_1m_2\\ 
D_2&=m_2m_5+m_1m_3+2m_1m_2-m_4m_5
\end{align}
The weight grading on the graph looks as follows in each of the chambers. Vertical position of the vertices indicates weight.
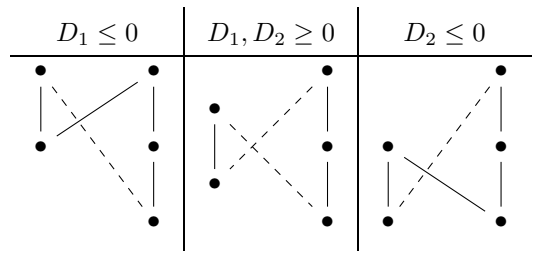
\begin{center}
\begin{tabular}{c|c|c}
$D_1\leq 0$ & 
$D_1,D_2\geq 0$ & $D_2\leq 0$ \\ \hline
\begin{tikzpicture}
\node (1) at (0,1) {$\bullet$};
\node (2) at (1.5,-1) {$\bullet$};
\node (3) at (1.5,0) {$\bullet$};
\node (4) at (1.5,1) {$\bullet$};
\node (5) at (0,0) {$\bullet$};
\draw[dashed] (1) edge (2);
\draw (2) edge (3);
\draw (3) edge (4);
\draw (4) edge (5);
\draw (5) edge (1);
\end{tikzpicture}
&
\begin{tikzpicture}
\node (1) at (0,.5) {$\bullet$};
\node (2) at (1.5,-1) {$\bullet$};
\node (3) at (1.5,0) {$\bullet$};
\node (4) at (1.5,1) {$\bullet$};
\node (5) at (0,-.5) {$\bullet$};
\draw[dashed] (1) edge (2);
\draw (2) edge (3);
\draw (3) edge (4);
\draw[dashed] (4) edge (5);
\draw (5) edge (1);
\end{tikzpicture}
&
\begin{tikzpicture}
\node (1) at (0,0) {$\bullet$};
\node (2) at (1.5,-1) {$\bullet$};
\node (3) at (1.5,0) {$\bullet$};
\node (4) at (1.5,1) {$\bullet$};
\node (5) at (0,-1) {$\bullet$};
\draw (1) edge (2);
\draw (2) edge (3);
\draw (3) edge (4);
\draw[dashed] (4) edge (5);
\draw (5) edge (1);
\end{tikzpicture}
\end{tabular}
\end{center}

One can make a variant of this example where the curve is not a closed loop but an embedded path with fixed endpoints at punctures. 
The quiver in this case is of type $A_n$ instead of extended $A_n$.

\bibliographystyle{plain}
\bibliography{balanced}

\Addresses

\end{document}